\documentclass[a4paper,11pt]{article}
\usepackage{makeidx}
\usepackage[mathscr]{eucal}
\usepackage[dvips]{color}
\usepackage{amssymb, amsfonts, amsmath, amsthm, graphics} 

\newtheorem{proposition}{Proposition}[section]
\newtheorem{theorem}{Theorem}[section]
\newtheorem{lemma}[proposition]{Lemma}

\numberwithin{equation}{section}

\title{Existence of a ground state and blow-up problem for a nonlinear Schr\"odinger equation with critical growth}
\author{Takafumi Akahori, Slim Ibrahim, Hiroaki Kikuchi and Hayato Nawa}
\date{}

\begin{document}
\maketitle


\section{Introduction}
\label{11/06/22/14:14}
In this paper, we consider the following nonlinear Schr\"odinger equation. 
\begin{equation}\tag{NLS}
\label{11/03/03/7:11}
2i\frac{\partial \psi}{\partial t}+\Delta \psi 
+
\mu |\psi|^{p-1}\psi 
+
|\psi|^{\frac{4}{d-2}}\psi
=0,
\end{equation}
where $\psi=\psi(x,t)$ is a complex-valued function on $\mathbb{R}^{d}\times \mathbb{R}$ ($d\ge 3$), $\Delta$ is the Laplace operator on $\mathbb{R}^{d}$, $\mu>0$ and $1+\frac{4}{d}<p< 2^{*}-1$ ($2^{*}:=\frac{2d}{d-2}$). 
\par 
The nonlinearity of (\ref{11/03/03/7:11}) is attractive and contains the energy-critical term $|\psi|^{\frac{4}{d-2}}\psi$ in the scaling sense. It is known (see \cite{Cazenave-Weissler1, Tao-Visan}) that: 
\\
{\rm (i)} for any $t_{0}\in \mathbb{R}$ and datum $\psi_{0}\in H^{1}(\mathbb{R}^{d})$, there exists a unique solution $\psi \in C(I_{\max},H^{1}(\mathbb{R}^{d}))$ to (\ref{11/03/03/7:11}) with $\psi(t_{0})=\psi_{0}$, where $I_{\max}$ denotes the maximal interval on which $\psi$ exists;  
\\
{\rm (ii)} the solution $\psi$ enjoys the following conservation laws:
\begin{align}
\label{11/06/14/11:02}
\left\| \psi(t) \right\|_{L^{2}}&=\left\| \psi_{0} \right\|_{L^{2}}
\qquad 
\mbox{for any $t \in I_{\max}$},
\\[6pt]
\label{11/06/14/11:03}
\mathcal{H}(\psi(t))&=\mathcal{H}(\psi_{0})
\qquad 
\mbox{for any $t \in I_{\max}$},
\end{align}
where 
\begin{equation}
\label{11/02/26/1:03}
\mathcal{H}(u)
:=
\left\|\nabla u \right\|_{L^{2}}^{2}
-\mu 
\frac{2}{p+1}\left\|u \right\|_{L^{p+1}}^{p+1}
-
\frac{d-2}{d}\left\| u \right\|_{L^{2^{*}}}^{2^{*}}.
\end{equation}
\par 
We are interested, like in \cite{Berestycki-Cazenave}, in the existence of standing wave and blowup problem for (\ref{11/03/03/7:11}).
Here, the standing wave means a solution to (\ref{11/03/03/7:11}) of the form $\psi(x,t)=e^{\frac{i}{2}t \omega}Q(x)$, so that $Q$  must satisfy the elliptic equation 
\begin{equation}\label{11/02/25/6:42}
- \Delta u + \omega u 
-\mu |u|^{p-1}u 
- |u|^{\frac{4}{d-2}}u = 0.
\end{equation}
When $\mu=0$, it is well-known that the equation (\ref{11/02/25/6:42}) has no solution for $\omega>0$. Moreover, if the subcritical perturbation is repulsive, i.e., $\mu <0$, then there is no non-trivial solution; Indeed, using the Pohozaev identity, we can verify that any solution $Q \in H^{1}(\mathbb{R}^{d})$\footnote{The solution  $Q \in H^{1}(\mathbb{R}^{d})$ also belongs to $L_{loc}^{\infty}(\mathbb{R}^{d})$ (see, e.g., Appendix B in \cite{Struwe}).} to (\ref{11/02/25/6:42}) with $\mu<0$ obeys that
\begin{equation}\label{11/02/26/9:30}
0
=
\omega \left\| Q \right\|_{L^{2}}^{2}
-\mu
\left( 1-\frac{d(p-1)}{2(p+1)}\right)\left\| Q \right\|_{L^{p+1}}^{p+1},
\end{equation}
which, together with $\omega>0$ and $p<2^{*}-1$, shows $Q$ is trivial. Thus, $\mu>0$ is necessary for the existence of non-trivial solution to (\ref{11/02/25/6:42}). 
\par     
Our first aim is to seek a special solution to (\ref{11/02/25/6:42}) called ground state via a certain variational problem; $Q$ is said to be the ground state, if it is a non-trivial solution to (\ref{11/02/25/6:42}) and  
\begin{equation}\label{11/06/14/11:45}
\mathcal{S}_{\omega}(Q)=\min\{ \mathcal{S}_{\omega}(u) \colon 
\mbox{$u$ is a solution to (\ref{11/02/25/6:42})}
\},
\end{equation}
where 
\begin{equation}
\label{11/02/25/6:43}
\mathcal{S}_{\omega}(u)
:=
\omega \left\|u \right\|_{L^{2}}^{2}
+
\mathcal{H}(u). 
\end{equation}
More precisely, we look for the ground state as a minimizer of the variational problem below:   
\begin{equation}\label{11/02/25/6:54}
m_{\omega}
:=
\inf\left\{ \mathcal{S}_{\omega}(u) 
\bigm|   u\in H^{1}(\mathbb{R}^{d})\setminus \{0\},
\ \mathcal{K}(u)=0
\right\},
\end{equation}
where 
\begin{equation}\label{11/02/25/6:44}
\begin{split}
\mathcal{K}(u)
&:=
\frac{d}{d\lambda}\mathcal{S}_{\omega}(T_{\lambda}u)\biggm|_{\lambda=1}
=
\frac{d}{d\lambda}\mathcal{H}(T_{\lambda}u)\biggm|_{\lambda=1}
\\[6pt]
&=
2\left\|\nabla u \right\|_{L^{2}}^{2}
-\mu
\frac{d(p-1)}{p+1}\left\|u \right\|_{L^{p+1}}^{p+1}
-
2 
\left\| u \right\|_{L^{2^{*}}}^{2^{*}},
\end{split}
\end{equation}
and $T_{\lambda}$ is the $L^{2}$-scaling operator defined by 
\begin{equation}
\label{11/02/25/8:02}
(T_{\lambda}u)(x):=\lambda^{\frac{d}{2}}u(\lambda x)
\qquad 
\mbox{for $\lambda>0$}.
\end{equation}
We can verify the following fact (see \cite{Berestycki-Cazenave, 
 LeCoz} for the proof): 
\begin{proposition}\label{11/06/15/16:19}
Any minimizer of the variational problem for $m_{\omega}$ becomes a ground state  of (\ref{11/02/25/6:42}). 
\end{proposition}

We remark that the existence of the ground state is studied by many authors (see, e.g., \cite{Alves-Souto-Montenegro, Berestycki-Lions, Berestycki-Cazenave}); In particular, in \cite{Alves-Souto-Montenegro}, they proved the existence of a ground state for a class  of elliptic equations including (\ref{11/02/25/6:42}) through a minimization problem different from ours. An advantage of our variational problem is that 
 we can prove the instability of the ground state found through it. Indeed, we can prove the blowup result (see Theorem \ref{11/04/22/10:26} below). 

\par 
In order to find the minimizer of the variational problem for $m_{\omega}$, we also consider the auxiliary variational problem:
\begin{equation}\label{11/02/25/6:55}
\widetilde{m}_{\omega}
=
\inf\left\{ \mathcal{I}_{\omega}(u) 
\bigm| u\in H^{1}(\mathbb{R}^{d})\setminus \{0\},
\ \mathcal{K}(u)\le 0
\right\},
\end{equation}
where 
\begin{equation}\label{11/02/25/6:45}
\begin{split}
\mathcal{I}_{\omega}(u)
&:=
\mathcal{S}_{\omega}(u)-\frac{2}{d(p-1)}\mathcal{K}(u)
\\[6pt]
&=
\omega \left\|u \right\|_{L^{2}}^{2}
+
\frac{d(p-1)-4}{d(p-1)}\left\|\nabla u \right\|_{L^{2}}^{2}
+\frac{4-(d-2)(p-1)}{d(p-1)}
\left\|u \right\|_{L^{2^{*}}}^{2^{*}}.
\end{split}
\end{equation}
Note that the subcritical potential term disappears in the functional $\mathcal{I}_{\omega}$. The problem (\ref{11/02/25/6:45}) is good to treat, since the functional $\mathcal{I}_{\omega}$ is positive and the constraint is $\mathcal{K}\le 0$ which is stable under the Schwarz symmetrization. Moreover, we have the following;

\begin{proposition}\label{11/04/09/15:32}
Assume $d\ge 3$. Then, for any $\omega>0$ and $\mu>0$, we have;
\\[6pt]
(i) $m_{\omega}=\widetilde{m}_{\omega}$
\\[6pt]
(ii) Any minimizer of the variational problem for $\widetilde{m}_{\omega}$ is also a minimizer for $m_{\omega}$, and vice versa.
\end{proposition}

We state our main theorems: 
\begin{theorem}\label{11/02/25/6:59}
Assume $d\ge 4$. Then, for any $\omega>0$ and $\mu>0$, there exists a minimizer of the variational problem for $m_{\omega}$. We also have $m_{\omega}=\widetilde{m}_{\omega}>0$.
\end{theorem}

\begin{theorem}\label{11/04/22/10:26}
Assume $d=3$. If $\mu$ is sufficiently small, then we have;
\\[6pt]
{\rm (i)} the variational problem for $m_{\omega}$ has no minimizer. 
\\[6pt]
{\rm (ii)} there is no ground state for the equation (\ref{11/02/25/6:42}). 
\end{theorem}

We see from Theorem \ref{11/02/25/6:59} together with Proposition \ref{11/06/15/16:19} that when $d\ge 4$, a ground state exists for any $\omega>0$ and $\mu>0$.\par 
We give the proofs of Proposition \ref{11/04/09/15:32} and Theorem \ref{11/02/25/6:59} in Section \ref{11/04/09/15:40}. The proof of Theorem \ref{11/02/25/6:59} is based on the idea of Brezis and Nirenberg \cite{Bre-Ni}. Since we consider the variational problem different from theirs, we need a little  different argument in the estimate of $m_{\omega}$ (see Lemma \ref{11/04/09/15:33}). We give the proof of Theorem \ref{11/04/22/10:26} in  Section \ref{11/04/21/1:27}. 
\par 
Next, we move to the blowup problem. Using the variational value $m_{\omega}$, we define the set $A_{\omega,-}$ by
\begin{equation}\label{11/06/14/9:47}
A_{\omega,-}:=\left\{ 
u \in H^{1}(\mathbb{R}^{d})
\bigm| 
\mathcal{S}_{\omega}(u)<m_{\omega},
\ 
\mathcal{K}(u)<0
\right\}.
\end{equation} 
We can see that $A_{\omega,-}$ is invariant under the flow defined by (\ref{11/03/03/7:11}) (see Lemma \ref{10/12/23/11:48} below). Furthermore, we have the following result: 
\begin{theorem}\label{11/06/14/9:50}
Assume $d\ge 4$, $\omega>0$ and $\mu>0$. Let $\psi$ be a solution to (\ref{11/03/03/7:11}) starting from $A_{\omega,-}$, and let $I_{\max}$ be the maximal interval on which $\psi$ exists. If $\psi$ is radially symmetric, then $I_{\max}$ is bounded.
\end{theorem}

We will prove Theorem \ref{11/06/14/9:50} in Section \ref{11/06/14/10:06}.
\par 

Besides the blowup result, we can also prove the scattering result; Put 
\begin{equation}\label{11/06/17/18:08}
A_{\omega,+}:= \left\{ 
u \in H^{1}(\mathbb{R}^{d})
\bigm| 
\mathcal{S}_{\omega}(u)<m_{\omega},
\ 
\mathcal{K}(u)>0
\right\}.
\end{equation} 
Then, we have
\begin{theorem}\label{11/06/17/18:10}
Assume $d\ge 5$, $\omega>0$ and $\mu>0$. Let $\psi$ be a solution to (\ref{11/03/03/7:11}) starting from $A_{\omega,+}$. Then, $\psi$ exists globally in time and satisfies  
\begin{equation}\label{11/06/17/18:17}
\psi(t) \in A_{\omega,+}
\qquad 
\mbox{for any $t \in \mathbb{R}$},
\end{equation}
and there exists $\phi_{+},\phi_{-} \in H^{1}(\mathbb{R}^{d})$ such that 
\begin{equation}\label{11/06/17/18:19}
\lim_{t\to +\infty}
\left\| \psi(t)-e^{\frac{i}{2}t\Delta}\phi_{+}\right\|_{H^{1}}
=
\lim_{t\to -\infty}
\left\| \psi(t)-e^{\frac{i}{2}t\Delta}\phi_{-}\right\|_{H^{1}}
=0.
\end{equation}
\end{theorem}
We will give the proof of Theorem \ref{11/06/17/18:10} in a forthcoming paper. In fact, we will prove the scattering result for a more general nonlinearity. 
\par 
Finally, we refer to relations of our work to the previous ones. 
 Theorems \ref{11/06/14/9:50} and \ref{11/06/17/18:10} are extensions of the blow-up and scattering results in \cite{Akahori-Nawa} to the equation including the energy-critical term. On the other hand, Kenig and Merle \cite{Kenig-Merle} proved the scattering and blow-up results, like Theorems \ref{11/06/14/9:50} and \ref{11/06/17/18:10} above, for the equation \eqref{11/03/03/7:11} without the energy-subcritical term in the radial case. The radial assumption  is removed by Killip and Visan \cite{Killip-Visan} in the 
 dimensions five and higher. The three and four dimensional cases  remain open, which relates with the restriction $d\ge 5$ in Theorem \ref{11/06/17/18:10} above.  
       
\section{Proofs of Proposition \ref{11/04/09/15:32} 
and Theorem \ref{11/02/25/6:59}}
\label{11/04/09/15:40}
In this section, we prove Proposition \ref{11/04/09/15:32} and Theorem \ref{11/02/25/6:59}. In the proofs, we often use the following easy facts:
\begin{equation}
\label{11/04/09/16:06}
\begin{split}
&\mbox{For any $u \in H^{1}(\mathbb{R}^{d})\setminus \{0\}$, there exists  a unique $\lambda(u)>0$ such that}
\\
&\qquad 
\mathcal{K}(T_{\lambda}u)
\left\{
\begin{array}{ll}
>0 & \mbox{for any $0< \lambda <\lambda(u)$},
\\
=0 & \mbox{for $\lambda=\lambda(u)$},
\\
<0 & \mbox{for any $\lambda >\lambda(u)$}.
\end{array} 
\right.
\end{split}
\end{equation}
\begin{equation}
\label{11/04/09/16:07}
\frac{d}{d\lambda}\mathcal{S}_{\omega}(T_{\lambda}u)
=
\frac{1}{\lambda} \mathcal{K}(T_{\lambda}u)
\quad 
\mbox{for any $u \in H^{1}(\mathbb{R}^{d})$ and $\lambda>0$}.
\end{equation}
Moreover, we see from (\ref{11/04/09/16:06}) and (\ref{11/04/09/16:07}) that for any $u \in H^{1}(\mathbb{R}^{d})\setminus \{0\}$, 
\begin{align}
\label{11/06/17/9:20}
&\mathcal{S}_{\omega}(T_{\lambda}u) 
< 
\mathcal{S}_{\omega}(T_{\lambda (u)}u)
\qquad 
\mbox{for any $\lambda \neq \lambda(u)$},
\\[6pt]
\label{11/06/17/9:21}
&\mbox{$\mathcal{S}_{\omega}(T_{\lambda}u)$ is concave for 
$\lambda > \lambda(u)$}.
\end{align}

Let us begin with the proof of Proposition \ref{11/04/09/15:32}. 
\begin{proof}[Proof of Proposition \ref{11/04/09/15:32}]
We first show the claim (i): $m_{\omega}=\widetilde{m}_{\omega}$. Let $\{u_{n}\}$ be any minimizing sequence for $\widetilde{m}_{\omega}$, i.e., $\{u_{n}\}$ is a sequence in $H^{1}(\mathbb{R}^{d})\setminus \{0\}$ such that \begin{align}
\label{11/04/09/16:00}
&\lim_{n\to \infty}\mathcal{I}_{\omega}(u_{n})=\widetilde{m}_{\omega},
\\[6pt]
\label{11/04/09/16:01}
&\mathcal{K}(u_{n})\le 0 
\quad 
\mbox{for any $n \in \mathbb{N}$}.
\end{align} 
Using (\ref{11/04/09/16:06}), for each $n \in \mathbb{N}$, we can take $\lambda_{n}\in (0, 1]$ such that $\mathcal{K}(T_{\lambda_{n}}u_{n})=0$. Then, we have 
\begin{equation}
\label{11/04/09/16:43}
m_{\omega}
\le 
\mathcal{S}_{\omega}(T_{\lambda_{n}}u_{n})
=
\mathcal{I}_{\omega}(T_{\lambda_{n}}u_{n})
\le 
\mathcal{I}_{\omega}(u_{n})
=
\widetilde{m}_{\omega}+o_{n}(1), 
\end{equation} 
where we have used $\lambda_{n}\le 1$ to derive the second inequality. Hence, we obtain $m_{\omega}\le \widetilde{m}_{\omega}$. 
On the other hand, it follows from   
\begin{equation}\label{11/03/04/4:55}
\widetilde{m}_{\omega}\le \inf_{{u\in H^{1}
\setminus \{0\}}\atop {\mathcal{K}=0}}\mathcal{I}_{\omega}(u)
=
\inf_{{u\in H^{1}
\setminus \{0\}}\atop {\mathcal{K}=0}}
\mathcal{S}_{\omega}(u)
=
m_{\omega}
\end{equation}
that $\widetilde{m}_{\omega}\le m_{\omega}$. Thus, we have proved that $m=m_{\omega}$. 
\par 
Next, we shall show the claim (ii). Let $Q$ be any minimizer for $\widetilde{m}_{\omega}$, i.e., $Q \in H^{1}(\mathbb{R}^{d})\setminus \{0\}$ with $\mathcal{K}(Q)\le 0$ and $\mathcal{I}_{\omega}(Q)=\widetilde{m}_{\omega}$. Since $\mathcal{I}_{\omega}=\mathcal{S}_{\omega}-\frac{2}{d(p-1)}\mathcal{K}$ (see (\ref{11/02/25/6:45})) and $m_{\omega}=\widetilde{m}_{\omega}$, it is sufficient to show that $\mathcal{K}(Q)=0$. Suppose the contrary that $\mathcal{K}(Q)<0$, so that there exists $0<\lambda_{0}<1$ such that $\mathcal{K}(T_{\lambda_{0}}Q)=0$. Then, we have 
\begin{equation}\label{11/04/30/9:41}
\widetilde{m}_{\omega}\le \mathcal{I}_{\omega}(T_{\lambda_{0}}Q)
< \mathcal{I}_{\omega}(Q)=\widetilde{m}_{\omega},
\end{equation}
which is a contradiction. Hence, $\mathcal{K}(Q)=0$ and therefore $Q$ is also a minimizer for $m_{\omega}$. 
\end{proof}

Now, employing the idea of  Br\'ezis and Nirenberg \cite{Bre-Ni},  we prove Theorem \ref{11/02/25/6:59}. 
To this end, we introduce another variational value $\sigma$:
\begin{equation}\label{11/04/09/17:01}
\sigma
:=
\inf\left\{
\left\| \nabla u \right\|_{L^{2}}^{2}
\bigm| 
u \in \dot{H}^{1}(\mathbb{R}^{d}), 
\ 
\left\| u \right\|_{L^{2^{*}}}=1
\right\}.
\end{equation}
Then, we have;
\begin{proposition}\label{11/06/17/16:40}
Assume that for  $\omega>0$ and $\mu>0$, we have 
\begin{equation}\label{11/06/17/16:42}
m_{\omega}<\frac{2}{d}\sigma^{\frac{d}{2}}.
\end{equation}
Then, there exists a minimizer of the variational problem for $m_{\omega}$. 
\end{proposition}
\begin{proof}[Proof of Proposition \ref{11/06/17/16:40}]
Let $\{u_{n}\}$ be a minimizing sequence of the variational problem  for $\widetilde{m}_{\omega}$. Extracting some subsequence, we may assume that 
\begin{equation}\label{11/04/10/17:40}
\mathcal{I}_{\omega}(u_{n}) \le 1+\widetilde{m}_{\omega}
\qquad 
\mbox{for any $n \in \mathbb{N}$}.
\end{equation}
We denote the Schwarz symmetrization of $u_{n}$ by $u_{n}^{*}$. Then, we easily  see that  
\begin{align}
\label{11/03/04/2:13}
&\mathcal{K}(u_{n}^{*})\le 0
\quad 
\mbox{for any $n \in \mathbb{N}$},
\\[6pt]
\label{11/03/04/1:30}
&\lim_{n\to \infty}\mathcal{I}_{\omega}(u_{n}^{*})=\widetilde{m}_{\omega},
\\[6pt]
\label{11/03/04/1:40}
& \sup_{n\in \mathbb{N}}
\left\| u_{n}^{*} \right\|_{H^{1}} < \infty
\quad 
\mbox{for any $n \in \mathbb{N}$}.
\end{align}
Since $\{u_{n}^{*}\}$ is  radially symmetric and bounded in $H^{1}(\mathbb{R}^{d})$, there exists a radially symmetric function $Q \in H^{1}(\mathbb{R}^{d})$ such that 
\begin{align}
\label{11/02/25/7:24}
&\lim_{n\to \infty}u_{n}^{*}=Q 
\quad 
\mbox{weakly in $H^{1}(\mathbb{R}^{d})$},
\\[6pt]
\label{11/04/10/17:52}
&\lim_{n\to \infty}u_{n}^{*}=Q 
\quad 
\mbox{strongly in $L^{q}(\mathbb{R}^{d})$ for $2<q<2^{*}$},
\\[6pt]
\label{11/03/04/1:49}
&\lim_{n\to \infty}u_{n}^{*}(x)=Q(x) 
\quad 
\mbox{for almost all $x \in \mathbb{R}^{d}$}.
\end{align}
This function $Q$ is a candidate for a minimizer of the variational problem (\ref{11/02/25/6:54}). We shall show that it is actually a minimizer.
\par 
Let us begin with the non-triviality of $Q$. Suppose the contrary that $Q$ is trivial. Then, it follows from (\ref{11/03/04/2:13}) and (\ref{11/04/10/17:52}) that 
\begin{equation}\label{11/04/10/18:14}
0\ge \limsup_{n\to \infty}\mathcal{K}(u_{n}^{*})
=2\limsup_{n\to \infty}\left\{ 
\left\| \nabla u_{n}^{*} \right\|_{L^{2}}^{2}
-\left\| u^{*} \right\|_{L^{2^{*}}}^{2^{*}}
\right\},
\end{equation}
so that   
\begin{equation}\label{11/04/10/18:18}
\limsup_{n\to \infty}\left\| \nabla u_{n}^{*} \right\|_{L^{2}}^{2}
\le 
\liminf_{n\to \infty}\left\| u_{n}^{*} \right\|_{L^{2^{*}}}^{2^{*}}.
\end{equation}
Moreover, this together with the definition of $\sigma$ (see (\ref{11/04/09/17:01})) gives us that 
\begin{equation}\label{11/04/10/18:22}
\limsup_{n\to \infty}\left\| \nabla u_{n}^{*} \right\|_{L^{2}}^{2}
\ge 
\sigma 
\liminf_{n\to \infty}\left\| u_{n}^{*} \right\|_{L^{2^{*}}}^{2}
\ge 
\sigma \limsup_{n\to \infty}
\left\| \nabla u_{n}^{*} \right\|_{L^{2}}^{\frac{2(d-2)}{d}},
\end{equation} 
so that 
\begin{equation}\label{11/04/18/8:57}
\sigma^{\frac{d}{2}}
\le \limsup_{n\to \infty}\left\| \nabla u_{n}^{*} \right\|_{L^{2}}.
\end{equation}
Hence, if $Q$ is trivial, then we have that 
\begin{equation}\label{11/04/10/18:26}
\begin{split}
\widetilde{m}_{\omega}
&=
\lim_{n\to \infty}\mathcal{I}_{\omega}(u_{n}^{*})
\\[6pt]
&\ge 
\lim_{n\to \infty}
\left\{ 
\frac{d(p-1)-4}{d(p-1)}\left\| \nabla u_{n}^{*} \right\|_{L^{2}}^{2}+
\frac{4-(d-2)(p-1)}{d(p-1)}\left\| u_{n}^{*} \right\|_{L^{2^{*}}}^{2^{*}}
\right\}
\\[6pt]
&\ge \frac{2}{d}\liminf_{n\to \infty}\left\| \nabla u_{n}^{*} \right\|_{L^{2}}^{2}
\ge \frac{2}{d}\sigma^{\frac{d}{2}}.
\end{split}
\end{equation}
However, this contradicts the hypothesis (\ref{11/06/17/16:42}). Thus, $Q$ is non-trivial. 
\par 
We shall show that $\mathcal{K}(Q)=0$. By the Brezis-Lieb Lemma (see \cite{Brezis-Lieb}), we have
\begin{equation}\label{11/02/25/7:29}
\mathcal{I}_{\omega}(u_{n}^{*})
-
\mathcal{I}_{\omega}(u_{n}^{*}-Q)
-
\mathcal{I}(Q)
=o_{n}(1).
\end{equation} 
Since $\displaystyle{\lim_{n\to \infty}\mathcal{I}_{\omega}(u_{n}^{*})= \widetilde{m}_{\omega}}$ and $\mathcal{I}_{\omega}$ is positive, we see from (\ref{11/02/25/7:29}) 
that
\begin{equation}\label{11/04/10/21:34}
\mathcal{I}_{\omega}(Q)\le \widetilde{m}_{\omega}.
\end{equation} 
\indent 
First, we suppose that $\mathcal{K}(Q)<0$, which together with the definition of $\widetilde{m}_{\omega}$ (see (\ref{11/02/25/6:55})) and (\ref{11/04/10/21:34}) implies that $\mathcal{I}_{\omega}(Q)=\widetilde{m}_{\omega}$. Moreover, we can take $\lambda_{0}\in (0,1)$ such that $\mathcal{K}(T_{\lambda_{0}}Q)=0$. Then, we have  
\begin{equation}\label{11/04/10/21:37}
\widetilde{m}_{\omega}\le \mathcal{I}_{\omega}(T_{\lambda_{0}}Q)<\mathcal{I}_{\omega}(Q)=\widetilde{m}_{\omega}.
\end{equation} 
This is a contradiction. Therefore $\mathcal{K}(Q)\ge 0$. 
\par 
Second, suppose the contrary that $\mathcal{K}(Q)>0$. Then, it follows 
from  (\ref{11/03/04/2:13}) and 
\begin{equation}\label{11/02/25/7:37}
\mathcal{K}(u_{n}^{*})-\mathcal{K}(u_{n}^{*}-Q)-\mathcal{K}(Q)
=o_{n}(1)
\end{equation}
that $\mathcal{K}(u_{n}^{*}-Q)<0$ for any sufficiently large $n \in \mathbb{N}$, so that there exists a unique $\lambda_{n}\in (0,1)$ such that 
$\mathcal{K}(T_{\lambda_{n}}(u_{n}^{*}-Q))=0$. Then, we easily verify that  
\begin{equation}\label{11/02/25/7:42}
\begin{split}
\widetilde{m}_{\omega}
&\le 
\mathcal{I}_{\omega}(T_{\lambda_{n}}(u_{n}^{*}-Q))
\\[6pt]
&\le \mathcal{I}_{\omega}(u_{n}^{*}-Q)
=\mathcal{I}_{\omega}(u_{n}^{*})
-
\mathcal{I}_{\omega}(Q)+o_{n}(1)
\\[6pt]
&= \widetilde{m}_{\omega} -
\mathcal{I}_{\omega}(Q)+o_{n}(1).
\end{split}
\end{equation}
Hence, we conclude $\mathcal{I}_{\omega}(Q)=0$. However, this contradicts the fact that $Q$ is non-trivial. Thus, $\mathcal{K}(Q)=0$. 
\par 
Since $Q$ is non-trivial and $\mathcal{K}(Q)=0$, we have 
\begin{equation}\label{11/02/25/7:57}
m_{\omega}
\le \mathcal{S}_{\omega}(Q)=\mathcal{I}_{\omega}(Q). 
\end{equation}
Moreover, it follows from (\ref{11/02/25/7:29}) and Proposition \ref{11/04/09/15:32} that 
\begin{equation}\label{11/04/10/21:53}
\mathcal{I}_{\omega}(Q) \le \liminf_{n\to \infty}\mathcal{I}_{\omega}(u_{n}^{*})
\le \widetilde{m}_{\omega}=m_{\omega}.
\end{equation} 
Combining (\ref{11/02/25/7:57}) and (\ref{11/04/10/21:53}), we obtain that $\mathcal{S}_{\omega}(Q)=\mathcal{I}_{\omega}(Q)=m_{\omega}$. Thus, we have proved that $Q$ is a minimizer for the variational problem (\ref{11/02/25/6:54}). 
\end{proof}

We shall give the proof of Theorem \ref{11/02/25/6:59}. In view of Proposition 
\ref{11/06/17/16:40}, it suffices to show the following lemma. 
\begin{lemma}
\label{11/04/09/15:33}
Assume $d\ge 4$. Then, for any $\omega>0$ and $\mu>0$, we have 
\begin{equation}\label{11/04/09/16:59}
m_{\omega}< \frac{2}{d}\sigma^{\frac{d}{2}}.
\end{equation} 
\end{lemma}
\begin{proof}[Proof of Lemma \ref{11/04/09/15:33}]
In view of Proposition \ref{11/04/09/15:32}, it suffices to show that 
\begin{equation}
\label{11/04/09/17:09}
\widetilde{m}_{\omega}<\frac{2}{d}\sigma^{\frac{d}{2}}. 
\end{equation}
Let $W$ be the Talenti function  
\begin{equation}\label{11/04/10/13:23}
W(x):=\left( \frac{\sqrt{d(d-2)}}{1+|x|^{2}}\right)^{\frac{d-2}{2}}.
\end{equation}
Put $W_{\varepsilon}(x):=\varepsilon^{-\frac{d-2}{2}}W(\varepsilon^{-1}x)$ for $\varepsilon>0$. It is known that  
\begin{align}
\label{11/06/17/9:09}
&-\Delta W -|W|^{\frac{4}{d-2}}W=0, 
\\[6pt]
\label{11/04/10/13:24}
&\sigma^{\frac{d}{2}}
=\left\| \nabla W_{\varepsilon} \right\|_{L^{2}}^{2}
=\left\| W_{\varepsilon} \right\|_{L^{2^{*}}}^{2^{*}}
\quad 
\mbox{for any $\varepsilon>0$}.
\end{align}
 
 We easily see that $W_{\varepsilon}$ belongs to $L^{2}(\mathbb{R}^{d})$ for $d\ge 5$. When $d=4$, $W_{\varepsilon}$  no longer belong to $L^{2}(\mathbb{R}^{4})$ and we need a cut-off function in our proof: 
Let $b \in C^{\infty}(\mathbb{R}^{4})$ be a function such that 
\begin{equation}\label{11/04/09/17:12}
b(x)=
\left\{ 
\begin{array}{ll}
1 & \mbox{if $|x|\le 1$},
\\[6pt]
0 & \mbox{if $|x|\ge 2$}.
\end{array} 
\right.
\end{equation}
We define the function $w_{\varepsilon}$ by 
\begin{equation}\label{11/04/09/17:17}
w_{\varepsilon}:=\left\{ \begin{array}{ll}
bW_{\varepsilon}&\mbox{when $d=4$},
\\[6pt]
W_{\varepsilon} & \mbox{when $d\ge 5$}.
\end{array}\right.
\end{equation}
Then, for any $\varepsilon <1$, we can verify (see p.35 of \cite{willem}) that 
\begin{align}
\label{11/04/09/17:22}
\left\| \nabla w_{\varepsilon}\right\|_{L^{2}}^{2}
&=
\sigma^{\frac{d}{2}}+O(\varepsilon^{d-2}),
\\[6pt]
\label{11/04/09/17:23}
\left\| w_{\varepsilon}\right\|_{L^{2^{*}}}^{2^{*}}
&= \sigma^{\frac{d}{2}}+O(\varepsilon^{d}),
\\[6pt]
\label{11/04/09/17:24}
\left\| w_{\varepsilon}\right\|_{L^{2}}^{2}
&= 
\left\{ \begin{array}{lcl}
C_{1}\varepsilon^{2}| \log{\varepsilon} |+O(\varepsilon^{2})
&\mbox{if}& d=4,
\\[6pt]
C_{1}\varepsilon^{2}+O(\varepsilon^{d-2})&\mbox{if}& d\ge 5,
\end{array}
\right.
\\[6pt]
\label{11/04/09/17:25}
\left\| w_{\varepsilon}\right\|_{L^{p+1}}^{p+1}
&= 
C_{2}\varepsilon^{d-\frac{(d-2)(p+1)}{2}}+O(\varepsilon^{\frac{(d-2)(p+1)}{2}}),
\end{align}
where $C_{1}$ and $C_{2}$ are positive constants independent of $\varepsilon$. Moreover, it follows from (\ref{11/04/09/16:06}) that; 
\begin{equation}\label{11/04/09/17:30}
\begin{split}
&\mbox{For any $\varepsilon\in (0,1)$, there exists a unique $\lambda_{\varepsilon}>0$ such that}
\\ 
&\qquad 
0=\mathcal{K}(T_{\lambda_{\varepsilon}}w_{\varepsilon})
=
2\lambda_{\varepsilon}
\left\| \nabla w_{\varepsilon} \right\|_{L^{2}}^{2}
-
\mu \frac{d(p-1)}{p+1}
\lambda_{\varepsilon}^{\frac{d(p-1)}{2}}
\left\| w_{\varepsilon} \right\|_{L^{p+1}}^{p+1}
-
2\lambda_{\varepsilon}^{2^{*}}
\left\| w_{\varepsilon} \right\|_{L^{2^{*}}}^{2^{*}}
. 
\end{split}
\end{equation}
Combining (\ref{11/04/09/17:22})--(\ref{11/04/09/17:25}) and (\ref{11/04/09/17:30}), we obtain that  
\begin{equation}\label{11/04/09/18:25}
\lambda_{\varepsilon}^{\frac{4}{d-2}}
=
1-C\lambda_{\varepsilon}^{\frac{d(p-1)-4}{2}}
\varepsilon^{d-\frac{(d-2)(p+1)}{2}}
+O(\varepsilon^{d-2}),
\end{equation}
where $C>0$ is some constant which is independent of $\varepsilon$. Since $\frac{4}{d-2}>\frac{d(p-1)}{2}-2$, we also find from (\ref{11/04/09/18:25}) that 
\begin{equation}\label{11/04/09/18:34}
\lim_{\varepsilon\downarrow 0}\lambda_{\varepsilon}=1.
\end{equation} 
Put $\lambda_{\varepsilon}=1+\alpha_{\varepsilon}$, where 
$\displaystyle{\lim_{\varepsilon\downarrow 1} \alpha_{\varepsilon} =0}$. Then, (\ref{11/04/09/18:25}) together with the expansion $(1+\alpha_{\varepsilon})^{q}=1+q\alpha_{\varepsilon} + o(\alpha_{\varepsilon})$ ($q>0$) gives us that 
\begin{equation}\label{11/04/09/18:47}
1+\frac{4}{d-2}\alpha_{\varepsilon}+o(\alpha_{\varepsilon})
=1
-C\left\{ 
1+{\frac{d(p-1)-4}{2}} \alpha_{\varepsilon}+o(\alpha_{\varepsilon})
\right\}\varepsilon^{d-\frac{(d-2)(p+1)}{2}}
+
O(\varepsilon^{d-2}),
\end{equation}
so that 
\begin{equation}\label{11/04/09/18:52}
\alpha_{\varepsilon}
= -C_{0}\varepsilon^{d-\frac{(d-2)(p+1)}{2}}+O(\varepsilon^{d-2})
\quad 
\mbox{for any sufficiently small $\varepsilon>0$},
\end{equation}
where $C_{0}>0$ is some constant which is independent of $\varepsilon$. When $d\ge 4$, $d-2>d-\frac{(d-2)(p+1)}{2}$ and hence we find from (\ref{11/04/09/18:52}) that
\begin{equation}\label{11/04/09/18:59}
\lambda_{\varepsilon}
= 1- C_{0}\varepsilon^{d-\frac{(d-2)(p+1)}{2}}+O(\varepsilon^{d-2})
\quad 
\mbox{for any sufficiently small $\varepsilon>0$}.
\end{equation}
Using (\ref{11/04/09/17:22})--(\ref{11/04/09/17:25}) and (\ref{11/04/09/18:59}), we obtain that 
\begin{equation}\label{11/04/09/17:03}
\begin{split}
\widetilde{m}_{\omega}
&\le \mathcal{I}_{\omega}(T_{\lambda_{\varepsilon}}w_{\varepsilon})\\[6pt]
&=
\omega \left\| w_{\varepsilon} \right\|_{L^{2}}^{2}
+\frac{d(p-1)-4}{d(p-1)}
\lambda_{\varepsilon}^{2}
\left\| \nabla w_{\varepsilon} \right\|_{L^{2}}^{2}
+
\frac{4-(d-2)(p-1)}{d(p-1)}
\lambda_{\varepsilon}^{2^{*}}
\left\| w_{\varepsilon} \right\|_{L^{2^{*}}}^{2^{*}}
\\[6pt]
&\le 
\omega C_{1}\varepsilon^{2}
+
\frac{d(p-1)-4}{d(p-1)}\lambda_{\varepsilon}^{2}
\sigma^{\frac{d}{2}}
+
\frac{4-(d-2)(p-1)}{d(p-1)}
\lambda_{\varepsilon}^{2^{*}}
\sigma^{\frac{d}{2}}+O(\varepsilon^{d-2})
\\[6pt]
&\le 
\omega C_{1}\varepsilon^{2}
+
\frac{d(p-1)-4}{d(p-1)}
\left\{ 1-2C_{0}\varepsilon^{d-\frac{(d-2)(p+1)}{2}} \right\}
\sigma^{\frac{d}{2}}
\\[6pt]
&\qquad 
+
\frac{4-(d-2)(p-1)}{d(p-1)}
\left\{ 1-2^{*}C_{0}\varepsilon^{d-\frac{(d-2)(p+1)}{2}} \right\}
\sigma^{\frac{d}{2}}
+o(\varepsilon^{d-\frac{(d-2)(p+1)}{2}})
\\[6pt]
&=
\frac{2}{d}\sigma^{\frac{d}{2}}
-\frac{16}{d(d-2)}C_{0}\varepsilon^{d-\frac{(d-2)(p+1)}{2}} 
+
\omega C_{1}\varepsilon^{2}
+o(\varepsilon^{d-\frac{(d-2)(p+1)}{2}})
\\[6pt]
&\hspace{180pt}
\qquad 
\mbox{for any sufficiently small $\varepsilon>0$}.
\end{split}
\end{equation}
Since $d-\frac{(d-2)(p+1)}{2}<2$ for $p>1$, we obtain $\widetilde{m}_{\omega}<\frac{2}{d}\sigma^{\frac{d}{2}}$.
\end{proof}

\section{Non-existence result in 3D}
\label{11/04/21/1:27}
In this section, we will give the proof of Theorem \ref{11/04/22/10:26}. Before proving the theorem, we prepare the following lemma.
\begin{lemma}\label{11/04/30/11:47}
Assume $d\ge 3$, $\omega>0$ and $\mu>0$. Let $u$ be a minimizer of the variational problem for $m_{\omega}$. Then, we have that 
\begin{equation}\label{11/04/30/11:48}
\left\| u \right\|_{H^{1}} \lesssim 1, 
\end{equation}
where the implicit constant is independent of $\mu$.
\end{lemma}
\begin{proof}[Proof of Lemma \ref{11/04/30/11:47}]
Fix a non-trivial function $\chi \in C_{c}^{\infty}(\mathbb{R}^{d})$ and put 
\begin{equation}\label{11/04/30/11:51}
\widetilde{\chi}:=
\left( 
\frac{ \left\| \nabla \chi \right\|_{L^{2}}^{2}}{\left\| \chi \right\|_{L^{2^{*}}}^{2^{*}}}
\right)^{\frac{1}{2^{*}-2}} \chi.
\end{equation}
It is easy to see that 
\begin{equation}\label{11/04/30/11:57}
\left\| \nabla \widetilde{\chi} \right\|_{L^{2}}^{2}
=
\left( \frac{ \left\| \nabla \chi \right\|_{L^{2}}^{2}}{\left\| \chi \right\|_{L^{2^{*}}}^{2^{*}}} \right)^{\frac{2}{2^{*}-2}}\left\|\nabla \chi \right\|_{L^{2}}^{2}
=
\left( 
\frac{\left\| \nabla \chi \right\|_{L^{2}}^{2}}
{\left\| \chi \right\|_{L^{2^{*}}}^{2^{*}}}
\right)^{\frac{2^{*}}{2^{*}-2}}
\left\| \chi \right\|_{L^{2^{*}}}^{2^{*}}
=
\left\| \widetilde{\chi} \right\|_{L^{2^{*}}}^{2^{*}},
\end{equation}
so that $\mathcal{K}(\widetilde{\chi})<0$. Hence, we obtain that 
\begin{equation}\label{11/04/30/12:17}
\widetilde{m}_{\omega}
\le 
\mathcal{I}_{\omega}(\widetilde{\chi}).
\end{equation}
Note here that $\mathcal{I}_{\omega}(\widetilde{\chi})$ is independent of $\mu$. 
\par 
Now, we take any minimizer $u$ for $m_{\omega}$. Then, Proposition \ref{11/04/09/15:32} shows that $u$ is also a minimizer for $\widetilde{m}_{\omega}$. Hence, we see from  (\ref{11/04/30/12:17}) that  
\begin{equation}\label{11/04/30/12:10}
\left\| u \right\|_{H^{1}}^{2}
\lesssim 
\mathcal{I}_{\omega}(u)=\widetilde{m}_{\omega}
\le \mathcal{I}_{\omega}(\widetilde{\chi}), 
\end{equation}
where the implicit constant is independent of $\mu$. Thus, we have completed the proof. 
\end{proof}

Now, we are in a position to prove Theorem \ref{11/04/22/10:26}. 
\begin{proof}[Proof of Theorem \ref{11/04/22/10:26}] 
We first prove the claim (i). Suppose the contrary that there exists a minimizer $u$ of the variational problem for $m_{\omega}$. It follows from Proposition \ref{11/04/09/15:32} that $u$ is also a minimizer for $\widetilde{m}_{\omega}$. Let $u^{*}$ be the Schwarz symmetrization of $u$. Then, we easily see that $u^{*}$ is a minimizer for $\widetilde{m}_{\omega}$ as well as $u$. Hence, using Proposition \ref{11/04/09/15:32} again, we find that $u^{*}$ is a radially symmetric solution to the elliptic equation (\ref{11/02/25/6:42}). Moreover, Lemma 1 in \cite{Berestycki-Lions} gives us that $u^{*} \in C^{2}(\mathbb{R}^{d})$. 
\par 
Put $\phi(r):=ru^{*}(r)$ ($r=|x|$). We see that $\phi$ satisfies the equation 
\begin{equation}\label{11/06/17/9:57}
\phi''-\omega \phi+\mu r^{1-p}\phi^{p}+r^{-4}\phi^{5}=0, 
\qquad 
r>0.
\end{equation}
Let $g \in W^{3,\infty}([0,\infty))$. Multiplying $g\phi'$ by the equation (\ref{11/06/17/9:57}) and integrating the resulting equation, we obtain  
\begin{equation}
\begin{split} \label{eqa-3}
& \quad
- \frac{g(0)}{2} (\phi'(0))^{2} 
- \frac{1}{2} \int_{0}^{\infty} g'(\phi')^{2}
+ \frac{1}{2} \int_{0}^{\infty} g' \phi^{2}   
\\[6pt]
& + \mu \int_{0}^{\infty} \left(\frac{p-1}{p+1}r^{-p}g -
\frac{1}{p+1}r^{1-p}g'\right)|\phi|^{p+1} +
\int_{0}^{\infty}
\left(\frac{2}{3}r^{-5}g-\frac{r^{-4}}{6}g'\right)|\phi|^{6} = 0.
\end{split}
\end{equation}
On the other hand, multiplying the equation (\ref{11/06/17/9:57}) by $g'\phi/2$ and integrating the resulting equation, we obtain
\begin{equation} \label{eqa-4}
- \frac{1}{2} \int_{0}^{\infty} g'(\phi')^{2} 
+ \frac{1}{4} \int_{0}^{\infty} g'''\phi^{2}
- \frac{1}{2} \int_{0}^{\infty} g' \phi^{2}
+ \frac{\mu}{2} \int_{0}^{\infty} r^{1-p}g' |\phi|^{p+1} 
+
\frac{1}{2} \int_{0}^{\infty} r^{-4}|\phi|^{6} g'  = 0.
\end{equation}
Subtracting \eqref{eqa-4} from \eqref{eqa-3},
we have the identity 
\begin{equation}\label{11/04/21/1:30}
\begin{split}
&- \frac{g(0)}{2} (\phi'(0))^{2}+\frac{1}{4}\int_{0}^{\infty}(g'''-4g')\phi^{2}
\\[6pt]
&= 
\mu 
\int_{0}^{\infty}
\left\{ \frac{p-1}{p+1}g - \frac{p+3}{2(p+1)}rg' \right\}
r^{-p}|\phi|^{p+1}
+\frac{2}{3}
\int_{0}^{\infty}\left\{ g-rg' \right\}r^{-5}
|\phi|^{6}.
\end{split}
\end{equation}
We take $g$ such that 
\begin{equation}\label{11/04/21/1:37}
\left\{ 
\begin{array}{rcl}
g'''(r)-4g'(r)&=&0, \quad r>0, 
\\[6pt]
g(0)&=&0.
\end{array}
\right.
\end{equation}
It is easy to verify that 
\begin{equation}\label{11/05/09/9:15}
g(r):=\frac{1-e^{-2r}}{2}
\ 
\mbox{is a solution to (\ref{11/04/21/1:37}).}
\end{equation}
Besides, we see from an elementary calculus that 
\begin{align}
\label{11/04/21/1:42}
g(r)-rg'(r)&=\frac{1-e^{-2r}}{2}-re^{-2r}\ge 0
\quad 
\mbox{for any $r>0$},
\\[6pt]
\label{11/04/30/17:12}
(g(r)-rg'(r))'&=2re^{-2r}>0
\quad 
\mbox{for any $r>0$}.
\end{align}
Taking this function, we see from the formula (\ref{11/04/21/1:30}) that 
\begin{equation}\label{11/04/26/21:47}
0=
\mu 
\int_{0}^{\infty}
\left\{ \frac{p-1}{p+1}g - \frac{p+3}{2(p+1)}rg' \right\}
r^{-p}|\phi|^{p+1}
+
\frac{2}{3}
\int_{0}^{\infty}\left\{ g-rg' \right\}r^{-5}
|\phi|^{6}.
\end{equation}
We shall show that 
\begin{align}
\label{11/04/27/13:52}
&\left| 
\int_{0}^{\infty}
\left\{ \frac{p-1}{p+1}g - \frac{p+3}{2(p+1)}rg' \right\}
r^{-p}|\phi|^{p+1}
\right|
\lesssim 1,
\\[6pt]
\label{11/04/30/17:00}
&\liminf_{\mu\to 0}\int_{0}^{\infty}
\left\{ g -   rg' \right\}
r^{-5}|\phi|^{6}
\gtrsim 1, 
\end{align}
where the implicit constants are independent of $\mu$. Once we obtain these estimates, taking $\mu\to 0$ in (\ref{11/04/21/1:30}), we derive a contradiction. Thus, the claim (i) holds.
\par 
We first prove (\ref{11/04/27/13:52}). It is easy to verify that the right-hand side of (\ref{11/04/27/13:52}) is bounded by 
\begin{equation}\label{11/04/27/14:36}
\begin{split}
\int_{0}^{\infty}(1-e^{-2r})r|u^{*}(r)|^{p+1}\,dr
+
\int_{0}^{\infty}e^{-2r}r^{2}|u^{*}(r)|^{p+1}\,dr
&\lesssim
\int_{0}^{\infty}r^{2}|u^{*}(r)|^{p+1}\,dr
\\[6pt]
&
\sim 
\left\| u \right\|_{L^{p+1}(\mathbb{R}^{3})}^{p+1}
\lesssim 
\left\| u \right\|_{H^{1}}^{p+1}
,
\end{split}
\end{equation}
where the implicit constants are independent of $\mu$. Combining  (\ref{11/04/27/14:36}) with Lemma \ref{11/04/30/11:47},  we obtain the estimate (\ref{11/04/27/13:52}). Next, we prove (\ref{11/04/30/17:00}). It follows from (\ref{11/04/21/1:42}), (\ref{11/04/30/17:12}) and the mean value theorem that 
\begin{equation}\label{11/04/30/17:13}
\int_{0}^{\infty}
\left\{ g -   rg' \right\}
r^{-5}|\phi|^{6}
=
\int_{0}^{\infty}
\left\{ g -   rg' \right\}
r|u^{*}|^{6}
\gtrsim  
\int_{0}^{\infty}r^{2}|u^{*}|^{6}
\sim \left\| u \right\|_{L^{6}}^{6},
\end{equation}
where the implicit constants are independent of $\mu$. Here, using the Sobolev embedding, $\mathcal{K}(u)=0$, the H\"older inequality and Lemma \ref{11/04/30/11:47}, we obtain that 
\begin{equation}\label{11/04/30/17:35}
\left\| u \right\|_{L^{6}}^{2} 
\lesssim 
\left\| \nabla u \right\|_{L^{2}}^{2}
\lesssim 
\left\| u \right\|_{L^{6}}^{\frac{3(p-1)}{2}}
+
\left\| u \right\|_{L^{6}}^{6} 
\quad 
\mbox{for all $0<\mu \le 1$}.
\end{equation}
Hence, we find that 
\begin{equation}\label{11/04/30/17:42}
1 \lesssim \left\| u \right\|_{L^{6}}^{6}
\quad \mbox{for all $0<\mu \le 1$}.
\end{equation}
The estimate (\ref{11/04/30/17:13}) together with (\ref{11/04/30/17:42}) gives (\ref{11/04/30/17:00}). 
\par 
We shall prove the claim (ii). Let $Q$ be a ground state of the equation (\ref{11/02/25/6:42}). We see from (\ref{11/06/14/11:45}) and Proposition \ref{11/06/15/16:19} that $Q$ is also a minimizer for $m_{\omega}$. However, this contradicts the claim (i). Thus, (ii) follows. 
\end{proof}

\section{Blowup result} 
\label{11/06/14/10:06}
In this section, we prove Theorem \ref{11/06/14/9:50}. The key lemma is the following. 
\begin{lemma}\label{10/12/23/11:48}
Assume  $d\ge 4$, $\omega>0$ and $\mu>0$. Let $\psi$ be a solution to (\ref{11/03/03/7:11}) starting from  $A_{\omega,-}$, and let $I_{\max}$ be the maximal interval where $\psi$ exists. Then, we have 
\begin{align}
\label{11/02/23/0:30}
& \psi(t) \in A_{\omega,-} 
\qquad \mbox{for any $t \in I_{\max}$}, 
\\[6pt]
\label{10/09/11/16:29}
&\sup_{t \in I_{\max}}\mathcal{K}(\psi(t))<0
.
\end{align}
\end{lemma}

\begin{proof}[Proof of Lemma \ref{10/12/23/11:48}]
Since the action $\mathcal{S}_{\omega}(\psi(t))$ is conserved with respect to $t$, we have $\mathcal{S}_{\omega}(\psi(t))<m_{\omega}$ for any $t \in I_{\max}$. Thus, it suffices for 
 (\ref{11/02/23/0:30}) to show that 
\begin{equation}\label{11/06/14/14:08}
\mathcal{K}(\psi(t))<0 
\qquad 
\mbox{for any $t \in I_{\max}$}. 
\end{equation}
Suppose the contrary that $\mathcal{K}(\psi(t))>0$ for some $t \in I_{\max}$. Then, it follows from the continuity of $\psi(t)$ in $H^{1}(\mathbb{R}^{d})$ that there exists $t_{0}\in I_{\max}$ such that $\mathcal{K}(\psi(t_{0}))=0$. Then, we see from the definition of $m_{\omega}$ that $\mathcal{S}_{\omega}(\psi(t_{0}))\ge m_{\omega}$, which is a contradiction. Thus, (\ref{11/06/14/14:08}) holds.
\par  
Next, we shall prove (\ref{10/09/11/16:29}). Since $\mathcal{K}(\psi(t))<0$ for any $t \in I_{\max}$, we see from (\ref{11/04/09/16:06}); 
\begin{equation}\label{11/06/17/10:09}
\mbox{For any $t \in I_{\max}$, there exists $0<\lambda(t)<1$ such that $\mathcal{K}(T_{\lambda(t)}\psi(t))=0$}.
\end{equation}
This together with the definition of $m_{\omega}$ shows that 
\begin{equation}\label{11/06/17/10:24}
\mathcal{S}_{\omega}(T_{\lambda(t)}\psi(t)) \ge m_{\omega}.
\end{equation}
Moreover, it follows from the concavity (\ref{11/06/17/9:21}) together with (\ref{11/04/09/16:07}), (\ref{11/06/14/14:08}) and (\ref{11/06/17/10:24}) that 
\begin{equation}\label{11/06/17/10:12}
\begin{split}
\mathcal{S}_{\omega}(\psi(t))
&>
\mathcal{S}_{\omega}(T_{\lambda(t)}\psi(t))
+
(1-\lambda(t))\frac{d}{d \lambda}
\mathcal{S}_{\omega}(T_{\lambda}\psi(t))\bigg|_{\lambda=1}
\\[6pt]
&=
\mathcal{S}_{\omega}(T_{\lambda(t)}\psi(t))
+
(1-\lambda(t))\mathcal{K}(\psi(t))
\\[6pt]
&> m_{\omega}+\mathcal{K}(\psi(t)).
\end{split}
\end{equation}
Fix $t_{0}\in I_{\max}$. Then, we have from 
(\ref{11/06/17/10:12}) and the conservation law of 
$\mathcal{S}_{\omega}$ that  
\begin{equation}\label{10/12/23/11:52}
\mathcal{K}(\psi(t))
< -m_{\omega}+\mathcal{S}_{\omega}(\psi(t))
=
-m_{\omega}+\mathcal{S}_{\omega}(\psi(t_{0})).
\end{equation}
Hence, (\ref{10/09/11/16:29}) follows from the hypothesis $\mathcal{S}_{\omega}(\psi(t_{0}))<m_{\omega}$. 
\end{proof}

Let us move to the proof of Theorem \ref{11/06/14/9:50}. 
\begin{proof}[Proof of Theorem \ref{11/06/14/9:50}]
The proof is based on the idea of \cite{Nawa8, Ogawa-Tsutsumi}. We introduce the generalized version of virial identity; Let $\rho$ be a smooth function on $\mathbb{R}$ such that 
\begin{align}
\label{10/09/20/17:34}
&\rho(x)=\rho(4-x) 
\qquad \mbox{for all $x \in \mathbb{R}$}, 
\\[6pt]
\label{10/09/20/17:35}
&\rho(x)\ge 0 
\qquad \mbox{for all $x \in \mathbb{R}$},
\\[6pt]
\label{10/09/20/17:36}
&\int_{\mathbb{R}}\rho(x)\,dx =1,
\\[6pt]
\label{10/09/20/17:37}
&{\rm supp}\,\rho \subset (1,3),
\\[6pt]
\label{10/09/20/17:38}
&\rho'(x) \ge 0 
\qquad 
\mbox{for all $x <2$}. 
\end{align}
Put 
\begin{align}
\label{10/09/20/17:41}
w(r)&:=r-\int_{0}^{r}(r-s)\rho(s)\,ds \qquad \mbox{for $r\ge 0$},
\\[6pt]
\label{10/09/20/17:43}
W_{R}(x)&:=R^{2}w\left( \frac{|x|^{2}}{R^{2}}\right)
\qquad
\mbox{for $x \in \mathbb{R}^{d}$ and $R>0$}.
\end{align}
Then, for any solution $\psi \in C(I_{\max},H^{1}(\mathbb{R}^{d}))$ to (\ref{11/03/03/7:11}) and $t_{0} \in I_{\max}$, putting $\psi_{0}:=\psi(t_{0})$, we have the identity 
\begin{equation}\label{08/03/29/19:05}
\begin{split}
&\int_{\mathbb{R}^{d}}W_{R}|\psi(t)|^{2}
\\[6pt]
&=
\int_{\mathbb{R}^{d}}W_{R}|\psi_{0}|^{2}
+
(t-t_{0})\Im \int_{\mathbb{R}^{d}}\nabla W_{R}\cdot \nabla \psi_{0}
\, \overline{\psi_{0}} 
+\int_{t_{0}}^{t}\int_{t_{0}}^{t'}\mathcal{K}(\psi(t''))\,dt''dt'
\\[6pt]
&\quad 
- \int_{t_{0}}^{t}\int_{t_{0}}^{t'}
\int_{\mathbb{R}^{d}}
\left\{ 
2\int_{0}^{\frac{|x|^{2}}{R^{2}}}\!\!\! \rho(r)\,dr |\nabla \psi(t'')|^{2}
+\frac{4|x|^{2}}{R^{2}}\rho \biggm(\frac{|x|^{2}}{R^{2}}\biggm)
\left| \frac{x}{|x|}\cdot\nabla \psi(t'')\right|^{2} 
\right\}
dt''dt'
\\[6pt]
&\quad 
+\int_{t_{0}}^{t}\int_{t_{0}}^{t'}
\!
 \int_{\mathbb{R}^{d}}
\!
\left\{ d \! \int_{0}^{\frac{|x|^{2}}{R^{2}}}\!\!\! \rho(r)\,dr 
+ 
\frac{2|x|^{2}}{R^{2}} \rho \bigg( \frac{|x|^{2}}{R^{2}}\bigg) 
\right\} 
\left\{ 
\mu
\frac{p-1}{p+1}\left| \psi(t'') \right|^{p+1}
\!+
\frac{2}{d}\left| \psi(t'') \right|^{2^{*}}
 \right\}
dt''dt'
\\[6pt]
&\quad 
-\frac{1}{4}
\int_{t_{0}}^{t}\int_{t_{0}}^{t'}
\int_{\mathbb{R}^{d}}
\Delta^{2}W_{R}\,|\psi(t'')|^{2}\,dt''dt'
\qquad \qquad 
\mbox{for any $t \in I_{\max}$ and $R>0$}.
\end{split}
\end{equation}
\noindent 
We easily verify that  
\begin{align}
\label{11/06/18/22:23}
\left|\nabla W_{R}(x) \right|^{2}
&\le 4  W_{R}(x) 
\qquad 
\mbox{for any $R>0$ and $x \in \mathbb{R}^{d}$}
\\[6pt]
\label{10/12/31/13:24}
\left\| W_{R} \right\|_{L^{\infty}}
&\lesssim 
R^{2},
\\[6pt]
\label{10/12/31/13:23}
\left\| \nabla W_{R} \right\|_{L^{\infty}}
&\lesssim R,
\\[6pt]
\label{10/12/27/18:17}
\left\| \Delta^{2}W_{R} \right\|_{L^{\infty}}
&\lesssim 
\frac{1}{R^{2}}\left\| \rho \right\|_{W^{2,\infty}(\mathbb{R})}.
\end{align}

Now, let $\psi$ be a solution to (\ref{11/03/03/7:11}) starting from $A_{\omega,-}$. We see from Lemma \ref{10/12/23/11:48} that 
\begin{equation}
\label{11/06/22/15:19}
\varepsilon_{0}:=-\sup_{t\in I_{\max}}\mathcal{K}(\psi(t))>0.
\end{equation}
Our aim is to show that the maximal existence interval $I_{\max}$ is bounded. We divide the proof into three steps.
\\[6pt]
{\bf Step 1}.  We claim that;  
\begin{equation} \label{11/06/18/17:12}
\begin{split}
&\mbox{There exists $m_{*}>0$ and $R_{*}>0$ such that for any $R\ge R_{*}$, }
\\[6pt]
&m_{*} < 
\inf
\left\{
\int_{|x| \ge R} |v(x)|^{2}\,dx 
\Biggm|  
\begin{array}{l}
v \in H_{rad}^{1}(\mathbb{R}^{d}), \quad  
\mathcal{K}^{R}(v) \le - \displaystyle{\frac{\epsilon_{0}}{4}}, 
\\[6pt]
\|v\|_{L^{2}} \leq \|\psi_{0}\|_{L^{2}}
\end{array}
\right\},
\end{split}
\end{equation}
where $H_{rad}^{1}(\mathbb{R}^{d})$ is the set of radially symmetric functions in $H^{1}(\mathbb{R}^{d})$, and 
\begin{equation}\label{11/06/18/14:59}
\begin{split}
\mathcal{K}^{R}(v)
&:= 
\int_{\mathbb{R}^{d}}
\left\{ 
2\int_{0}^{\frac{|x|^{2}}{R^{2}}}\!\!\! 
\rho(r)\,dr
+\frac{4|x|^{2}}{R^{2}}\rho \biggm(\frac{|x|^{2}}{R^{2}}\biggm)
\right\}
|\nabla v|^{2}
\\[6pt]
&\quad 
- \int_{\mathbb{R}^{d}}
\!
\left\{ d \! \int_{0}^{\frac{|x|^{2}}{R^{2}}}\!\!\! \rho(r)\,dr 
+ 
\frac{2|x|^{2}}{R^{2}} \rho \bigg( \frac{|x|^{2}}{R^{2}}\bigg) 
\right\} 
\left\{ 
\mu
\frac{p-1}{p+1}\left| v \right|^{p+1}
\!+
\frac{2}{d}\left| v \right|^{2^{*}}
 \right\}.
\end{split}
\end{equation}
Let $R>0$ be a sufficiently large constant to be chosen later, and let $v$ be a function such that  
\begin{align}
\label{11/06/19/11:33}
&v \in H_{rad}^{1}(\mathbb{R}^{d}),
\\[6pt]
\label{11/06/18/17:04}
&\mathcal{K}^{R}(v) \leq - \frac{\epsilon_{0}}{4}, 
\\[6pt] 
\label{11/06/18/17:06}
&\|v\|_{L^{2}} \leq \|\psi_{0}\|_{L^{2}}. 
\end{align}
Then, we see from \eqref{11/06/18/17:04} that 
\begin{equation}\label{11/06/18/17:09}
\begin{split}
&\frac{\epsilon_{0}}{4} 
+ 
\int_{\mathbb{R}^{d}}
\left\{
2\int_{0}^{\frac{|x|^{2}}{R^{2}}}
\rho(r)\,dr
+\frac{4|x|^{2}}{R^{2}}\rho \biggm(\frac{|x|^{2}}{R^{2}}\biggm)
\right\}
|\nabla v|^{2}
\\[6pt]
& 
\le - \mathcal{K}^{R}(v) 
+
\int_{\mathbb{R}^{d}}
\left\{
2\int_{0}^{\frac{|x|^{2}}{R^{2}}}
\rho(r)\,dr
+\frac{4|x|^{2}}{R^{2}}\rho \biggm(\frac{|x|^{2}}{R^{2}}\biggm)
\right\}|\nabla v|^{2}
\\[6pt]
& 
=
\int_{\mathbb{R}^{d}}
\!
\left\{ d \! \int_{0}^{\frac{|x|^{2}}{R^{2}}}\!\!\! \rho(r)\,dr 
+ 
\frac{2|x|^{2}}{R^{2}} \rho \bigg( \frac{|x|^{2}}{R^{2}}\bigg) 
\right\} 
\left\{ 
\mu
\frac{p-1}{p+1}\left| v \right|^{p+1}
\!+
\frac{2}{d}\left| v \right|^{2^{*}}
 \right\}.
\end{split}   
\end{equation}
To estimate the right-hand side of \eqref{11/06/18/17:09}, we employ the following inequality (see Lemma 6.5.11 in \cite{Cazenave}); Assume that $d\ge 1$. Let $\kappa$ be a non-negative and radially symmetric function in $C^{1}(\mathbb{R}^{d})$ with $|x|^{-(d-1)}\max\{-\frac{x\cdot\nabla \kappa }{|x|},\, 0\} \in L^{\infty}(\mathbb{R}^{d})$. Then, we have 
\begin{equation}\label{08/06/12/7:28}
\begin{split}
&
\|\kappa^{\frac{1}{2}} u\|_{L^{\infty}}
\\[6pt]
&\lesssim   
\|u\|_{L^{2}}^{\frac{1}{2}}
\left\{ 
\left\| |x|^{-(d-1)}\kappa \, \frac{x \cdot \nabla u}{|x|} \right\|_{L^{2}}^{\frac{1}{2}}
+
\left\| |x|^{-(d-1)}
\max \left\{-\frac{x \cdot \nabla \kappa}{|x|}, \ 0 \right\}
\right\|_{L^{\infty}}^{\frac{1}{2}}
\left\| u \right\|_{L^{2}}^{\frac{1}{2}}
\right\}
\\[6pt]
&\hspace{9cm}
\mbox{for any $u \in H_{rad}^{1}(\mathbb{R}^{d})$}, 
\end{split}
\end{equation}
where the implicit constant depends only on $d$. 
\par 
Now, put 
\begin{align}
\label{11/06/19/12:49}
\kappa_{1}(x)
&:= 
2\int_{0}^{\frac{|x|^{2}}{R^{2}}}\!\!\! 
\rho(r)\,dr
+\frac{4|x|^{2}}{R^{2}}\rho \biggm(\frac{|x|^{2}}{R^{2}}\biggm),
\\[6pt]
\label{11/06/19/11:56}
\kappa_{2}(x)
&:=
d \! \int_{0}^{\frac{|x|^{2}}{R^{2}}}\!\!\! \rho(r)\,dr 
+ 
\frac{2|x|^{2}}{R^{2}} \rho \bigg( \frac{|x|^{2}}{R^{2}}\bigg),
\end{align}
so that 
\begin{equation}\label{11/06/19/12:50}
\mathcal{K}^{R}(v)=
\int_{\mathbb{R}^{d}}\kappa_{1}|\nabla v|^{2}
-
\int_{\mathbb{R}^{d}}\kappa_{2}
\left(
\mu
\frac{p-1}{p+1}\left| v \right|^{p+1}
\!+
\frac{2}{d}\left| v \right|^{2^{*}}
 \right).
\end{equation}
Then, we can verify that  
\begin{align}
\label{11/06/19/12:26}
&{\rm supp}\,\kappa_{j} \subset \{|x|\ge R\} 
\qquad 
\mbox{for $j=1,2$}, 
\\[6pt]
\label{11/06/19/12:27}
&\left\| \kappa_{j} \right\|_{L^{\infty}}\lesssim 1
\qquad 
\mbox{for $j=1,2$},
\\[6pt]
\label{11/06/19/12:28}
&
\sup_{x\in \mathbb{R}^{d}}
\max
\left\{ 
-\frac{x\cdot \nabla \sqrt{\kappa_{2}(x)}}{|x|}, \, 0
\right\}
\lesssim \frac{1}{R},
\\[6pt]
\label{11/06/19/12:53}
&1\lesssim \inf_{|x|\ge R}\frac{\kappa_{1}(x)}{\kappa_{2}(x)}.
\end{align}
The inequality \eqref{08/06/12/7:28} together with 
 \eqref{11/06/19/12:27}, \eqref{11/06/19/12:28} and 
 \eqref{11/06/18/17:06} yields that
\begin{equation}\label{11/06/19/11:55}
\begin{split}
&\int_{\mathbb{R}^{d}}\kappa_{2} \frac{2}{d}|v|^{2^{*}} 
\le 
\|\sqrt{\kappa_{2}}^{\frac{1}{2}}v \|_{L^{\infty}}^{\frac{4}{d-2}}
\int_{\mathbb{R}^{d}}\kappa_{2}^{\frac{d-3}{d-2}} |v|^{2}
\\[6pt]
&
\lesssim 
\frac{1}{R^{\frac{2(d-1)}{d-2}}}
\left\| v \right\|_{L^{2}(|x|\ge R)}^{\frac{2}{d-2}}
\left\{ 
\left\|\sqrt{\kappa_{2}}\nabla v \right\|_{L^{2}}^{\frac{2}{d-2}}
+
\frac{1}{R^{\frac{2}{d-2}}}
\left\| v \right\|_{L^{2}}^{\frac{2}{d-2}}
\right\} \|v\|_{L^{2}}^{2}
\\[6pt]
&\le 
\frac{ \|\psi_{0}\|_{L^{2}}^{2}}{R^{\frac{2(d-1)}{d-2}}}
\left( 
\left\| v \right\|_{L^{2}(|x|\ge R)}
\left\|\sqrt{\kappa_{2}}\nabla v \right\|_{L^{2}}
\right)^{\frac{2}{d-2}}
+
\frac{ \|\psi_{0}\|_{L^{2}}^{2^{*}}}{R^{2^{*}}},
\end{split}
\end{equation}
where the implicit constant depends only on $d$ and $\rho$.
Moreover, applying the Young inequality ($ab \le \frac{a^{r}}{r}+\frac{b^{r'}}{r'}$ with $\frac{1}{r}+\frac{1}{r'}=1$)  to the right-hand side of \eqref{11/06/19/11:55}, we obtain that 
\begin{equation}\label{11/06/19/12:41}
\int_{\mathbb{R}^{d}}\kappa_{2} \frac{2}{d}|v|^{2^{*}} 
\lesssim  
\left\| v \right\|_{L^{2}(|x|\ge R)}^{2}
\left\|\sqrt{\kappa_{2}}\nabla v \right\|_{L^{2}}^{2}
+
\left(
\frac{\left\|\psi_{0}\right\|_{L^{2}}^{2}}{R^{\frac{2(d-1)}{d-2}}}
\right)^{\frac{d-2}{d-3}}
+
\frac{ \|\psi_{0}\|_{L^{2}}^{2^{*}}}{R^{2^{*}}}.
\end{equation}
Similarly, we have 
\begin{equation}\label{11/06/19/14:12}
\int_{\mathbb{R}^{d}}\kappa_{2} \frac{2}{d}|v|^{p+1} 
\lesssim  
\left\| v \right\|_{L^{2}(|x|\ge R)}^{2}
\left\|\sqrt{\kappa_{2}}\nabla v \right\|_{L^{2}}^{2}
+
\left(
\frac{\left\|\psi_{0}\right\|_{L^{2}}^{2}}
{R^{\frac{(d-1)(p-1)}{2}}}
\right)^{\frac{4}{5-p}}
+
\frac{ \|\psi_{0}\|_{L^{2}}^{2^{*}}}{R^{\frac{d(p-1)}{2}}}.
\end{equation}
If $R$ is so large that 
\begin{equation}\label{11/06/19/14:16}
\left(
\frac{\left\|\psi_{0}\right\|_{L^{2}}^{2}}
{R^{\frac{(d-1)(p-1)}{2}}}
\right)^{\frac{4}{5-p}}
+
\frac{ \|\psi_{0}\|_{L^{2}}^{2^{*}}}{R^{\frac{d(p-1)}{2}}}
+
\left(
\frac{\left\|\psi_{0}\right\|_{L^{2}}^{2}}{R^{\frac{2(d-1)}{d-2}}}
\right)^{\frac{d-2}{d-3}}
+
\frac{ \|\psi_{0}\|_{L^{2}}^{2^{*}}}{R^{2^{*}}}
\ll 
\varepsilon_{0},
\end{equation} 
then \eqref{11/06/18/17:09} together with  \eqref{11/06/19/12:41} and \eqref{11/06/19/14:16} shows that 
\begin{equation}\label{11/06/19/14:18}
\int_{\mathbb{R}^{d}} \left( \kappa_{1}
-
C \left\| v \right\|_{L^{2}(|x|\ge R)}^{2}\kappa_{2}
\right)
|\nabla v|^{2}
\le 
-\frac{\varepsilon_{0}}{8}
\end{equation}
for some constant $C>0$ depending only on $d$, $\mu$, $p$ and $\rho$ (the notation $A \ll \varepsilon_{0}$ means that the quantity $A$ is much smaller than $\varepsilon_{0}$). 
Hence, we conclude that  
\begin{equation}\label{11/06/19/14:26}
\inf_{|x|\ge R}{ \frac{\kappa_{1}(x)}{\kappa_{2}(x)}}< 
C\|v\|_{L^{2}(|x| \ge R)}^{2},
\end{equation}
which together with \eqref{11/06/19/12:53} gives us the desired result.
\\[6pt] 
{\bf Step 2}. Let $m_{*}$ and $R_{*}$ be constants found in (\ref{11/06/18/17:12}), and let $T_{\max}:=\sup{I_{\max}}$. Then, our next claim is that 
\begin{equation}\label{11/06/19/17:44}
\sup_{t\in [t_{0}, T_{\max})} \int_{|x| \ge R} 
|\psi(x,t)|^{2} \,dx \le m_{*}
\end{equation}
for any $R>R_{*}$ satisfying the following properties:
\begin{align}
\label{11/06/18/18:33}
&
\left\| \Delta^{2}W_{R}\right\|_{L^{\infty}} \|\psi_{0}\|^{2}_{L^{2}}\le \varepsilon_{0},
\\[6pt]
\label{11/06/18/18:34}
&
\int_{|x| \geq R} |\psi_{0}(x)|^{2}\,dx < m_{*}, 
\\[6pt]
\label{11/06/18/18:35}
&
\frac{1}{R^{2}} 
\left(1 + \frac{1}{\varepsilon_{0}}
\left\| \nabla \psi_{0} \right\|_{L^{\infty}}^{2} 
\right) 
\int_{\mathbb{R}^{d}} W_{R}|\psi_{0}|^{2} < m_{*}.  
\end{align}
Here, we remark that  it is possible to take $R$ satisfying \eqref{11/06/18/18:35} (see \cite{Nawa8}).
\par 
In order to prove \eqref{11/06/19/17:44}, we introduce 
\begin{equation}
\label{11/06/18/21:02}
T_{R} := \sup\left\{
T>  t_{0} \Biggm|  
\sup_{t_{0} \le t< T} \int_{|x| \geq R} |\psi(x,t)|^{2}\,dx \le m_{*}
\right\}
\quad 
\mbox{for $R>0$}, 
\end{equation}
and prove that $T_{R}=T_{\max}$ for any $R>0$ satisfying 
 \eqref{11/06/18/18:33}--\eqref{11/06/18/18:35}. 
\par 
It follows from (\ref{11/06/18/18:34}) together with the continuity of $\psi(t)$ in $L^{2}(\mathbb{R}^{d})$ that $T_{R}>t_{0}$. 
\par 
We suppose the contrary that $T_{R}< T_{\max}$. Then, we have 
\begin{equation} \label{11/06/18/21:11}
\int_{|x| \ge R} |\psi(x,T_{R})|^{2}\,dx = m_{*}. 
\end{equation}
Hence, we see from the definition of $m_{*}$ (see \eqref{11/06/18/17:12}) together with the mass conservation law (\ref{11/06/14/11:02}) that 
\begin{equation} 
\label{11/06/18/21:13}
- \frac{\varepsilon_{0}}{4} \leq 
\mathcal{K}^{R}(\psi(T_{R})). 
\end{equation}
Combining the generalized virial identity \eqref{08/03/29/19:05}   with \eqref{11/06/18/21:13} and \eqref{11/06/18/18:33}, 
 we obtain 
\begin{equation}
\label{11/06/18/21:22}
\begin{split}
\int_{\mathbb{R}^{d}}W_{R}|\psi(T_{R})|^{2} 
&
< 
\int_{\mathbb{R}^{d}}W_{R}|\psi_{0}|^{2} 
+ 
T_{R}\Im \int_{\mathbb{R}^{d}} \nabla W_{R} 
\cdot \nabla \psi_{0} \overline{\psi_{0}} 
- \frac{T_{R}^{2}}{2}\varepsilon_{0}  
\\[6pt]
& \qquad 
+ 
\frac{T^{2}_{R}}{8}  \varepsilon_{0}
+
\frac{T_{R}^{2}}{8} \varepsilon_{0}
\\[6pt]
&
=
\int_{\mathbb{R}^{d}}W_{R}|\psi_{0}|^{2}
+
\frac{1}{\varepsilon_{0}}
\left| \int_{\mathbb{R}^{d}} \nabla W_{R} 
\cdot \nabla \psi_{0} \overline{\psi_{0}} \right|^{2}
\\[6pt]
&\qquad 
-
\frac{\varepsilon_{0}}{4}
\left( T_{R}
-
\frac{2}{\varepsilon_{0}}\Im \int_{\mathbb{R}^{d}} \nabla W_{R} 
\cdot \nabla \psi_{0} \overline{\psi_{0}} 
\right)^{2}
\\[6pt]
&\le  
\left( 1 + 
\frac{1}{\varepsilon_{0}}
\left\| \nabla W_{R} \right\|_{L^{\infty}}^{2}
\right)
\int_{\mathbb{R}^{d}} W_{R}|\psi_{0}|^{2}.
\end{split}
\end{equation}
Hence, we see from \eqref{11/06/18/18:35} that 
\begin{equation}\label{11/06/18/22:17}
\int_{\mathbb{R}^{d}}W_{R}|\psi(T_{R})|^{2} 
<R^{2}m_{*}.  
\end{equation}
Since $W_{R}(x) \geq R^{2}$ for $|x| > R$, \eqref{11/06/18/22:17} gives us that
\begin{equation}
\label{11/06/19/22:42}
\int_{|x| \geq R} |\psi(x, T_{R})|^{2} dx 
= \frac{1}{R^{2}} \int_{|x|\geq R} R^{2} 
|\psi(x, T_{R})|^{2} dx 
\leq \frac{1}{R^{2}} \int_{\mathbb{R}^{d}}W_{R}|\psi(T_{r})|^{2} 
< m_{*}, 
\end{equation}
which contradicts \eqref{11/06/18/21:11}. Thus, we have proved $T_{R}=T_{\max}$  and therefore (\ref{11/06/19/17:44}) holds. 
\par 
Similarly, putting $T_{\min}:=\inf{I_{\max}}$, we have  
\begin{equation}\label{11/06/19/22:39}
\sup_{t\in (T_{\min}, t_{0}]} \int_{|x| \ge R} 
|\psi(x,t)|^{2} \,dx \le m_{*}
\end{equation}
for any $R>R_{*}$ satisfying \eqref{11/06/18/18:33}--\eqref{11/06/18/18:35}.
\\[6pt]
{\bf Step 3} We complete the proof of Theorem \ref{11/06/14/9:50}. Take $R>R_{*}$ satisfying 
 \eqref{11/06/18/18:33}--\eqref{11/06/18/18:35}. Then, it follows from the definition of $m_{*}$ together with \eqref{11/06/19/17:44} and \eqref{11/06/19/22:39} that 
\begin{equation}\label{11/06/19/18:05}
- \frac{\epsilon_{0}}{4} 
\le 
\mathcal{K}^{R}(\psi(t))
\qquad
\mbox{for any $t \in I_{\max}$}. 
\end{equation}
Hence, we see from the generalized virial identity \eqref{08/03/29/19:05} together with \eqref{11/06/19/18:05} and \eqref{11/06/18/18:33} that   
\begin{equation}
\label{11/06/18/22:53}
\int_{\mathbb{R}^{d}}
W_{R}|\psi(t)|^{2} 
\le 
\int_{\mathbb{R}^{d}}W_{R} |\psi_{0}|^{2} 
+ 
t\Im \int_{\mathbb{R}^{d}}\nabla W_{R}\cdot \nabla \psi_{0}
\overline{\psi_{0}} 
- 
\frac{\varepsilon_{0}t^{2}}{2}
\quad
\mbox{for any $t \in I_{\max}$}.
\end{equation}
Supposing $\sup{I_{\max}}= +\infty$ or $\inf{I_{\max}}=-\infty$, we have from \eqref{11/06/18/22:53} that 
\begin{equation}\label{11/06/24/17:58}
\int_{\mathbb{R}^{d}}
W_{R}|\psi(t)|^{2}<0 
\qquad 
\mbox{for some $t \in I_{\max}$}, 
\end{equation} 
which is a contradiction. Thus, $I_{\max}$ is bounded. 
\end{proof}

\bibliographystyle{plain}

\begin{thebibliography}{99}
\bibitem{Akahori-Nawa}
Akahori, T. and Nawa, H., Blowup and Scattering problems for the Nonlinear Schr\"odinger equations, preprint, arXiv:1006.1485.

\bibitem{Alves-Souto-Montenegro}
 Alves, C. O., Souto, M.A.S., and 
 Montenegro, M., 
 Existence of a ground state solution for a nonlinear scalar field equation with critical growth.
 Calc. Var. 

\bibitem{Berestycki-Lions}
 Berestycki, H. and Lions, P.-L. Nonlinear scalar field equations. I. Existence of a ground state. Arch. Rational Mech. Anal. {\bf 82} (1983), 313--345. 



\bibitem{Berestycki-Cazenave}
 Berestycki, H. and Cazenave, T., Instabilit\'e des \'etats stationnaires dans les \'equations de Schr\"odinger et de Klein-Gordon non lin\'eaires. C. R. Acad. Sci. Paris S\'er. I Math. {\bf 293} (1981), 489--492. 


\bibitem{Brezis-Lieb}
 Br\'ezis, H. and Lieb, E.H.,
 A relation between pointwise convergence of functions and convergence of functionals.
 Proc. Amer. Math. Soc. {\bf 88} (1983) 485--489. 

\bibitem{Bre-Ni}
  Brezis, H. and Nirenberg, L., 
 Positive solutions of nonlinear elliptic equations involving critical 
 Sobolev exponents, 
 Comm. Pure Appl. Math. {\bf 36} (1983), 437--477. 

\bibitem{Cazenave}
 Cazenave, T., 
 Semilinear Schr\"odinger equations. 
 Courant Lecture Notes in Mathematics, {\bf 10}.
 New York university, Courant Institute of Mathematical Sciences, New York        (2003) .

\bibitem{Cazenave-Weissler1} 
 Cazenave, T. and Weissler, F.B.,
 The Cauchy problem for the critical nonlinear Schr\"odinger equation in $H^{s}$, Nonlinear Anal. {\bf 14} (1990), 807--836.

\bibitem{Kenig-Merle}
 Kenig, C.E. and  Merle, F.,  
 Global well-posedness, scattering and blow-up for the energy-critical, focusing, non-linear Schr\"odinger equation in the radial case. 
 Invent. Math. {\bf 166} (2006), 645--675.

\bibitem{Killip-Visan}
 Killip R. and Visan M.,  
 The focusing energy-critical nonlinear Schr\"odinger equation in dimensions five and higher. 
 Amer. J. Math. {\bf 132} (2010) 361-424.

\bibitem{LeCoz}
 Le Coz, S., 
 A note on Berestycki-Cazenave's classical instability result for nonlinear 
 Schr\"odinger equations.
 Adv. Nonlinear Stud. {\bf 8} (2008), 455--463. 



\bibitem{Nawa8} Nawa, H., 
 Asymptotic and Limiting Profiles of Blowup Solutions of the Nonlinear Schr\"odinger Equations with Critical Power, 
 Comm. Pure Appl. Math. {\bf 52} (1999), 193-270.

\bibitem{Ogawa-Tsutsumi}
 Ogawa, T. and Tsutsumi, Y.,
 Blow-up of $H^1$-solution for the nonlinear Schr\"odinger equation.
 J. Differential Equations {\bf 92} (1991) 317--330.

\bibitem{Struwe}
 Struwe, M. 
 Variational methods. Applications to nonlinear partial differential equations and Hamiltonian systems. Fourth edition. 
 Ergebnisse der Mathematik und ihrer Grenzgebiete. 3. Springer-Verlag, Berlin.

\bibitem{Tao-Visan}
 Tao, T. and Visan, M., 
 Stability of energy-critical nonlinear Schr\"odinger equations in high dimensions, 
 Electron. J. Differential Equations {\bf 118} (2005), 1--28.

\bibitem{willem}
 Willem, M., 
 Minimax theorems, 
 Birkh\"{a}user, 1996. 
\end{thebibliography}

\noindent
{
Takafumi Akahori,
\\
Faculty of Engineering
\\
Shizuoka University, 
\\ 
Jyohoku 3-5-1, Hamamatsu, 432-8561, Japan 
\\ 
E-mail: ttakaho@ipc.shizuoka.ac.jp
}
\\
\\
{
Slim Ibrahim,
\\ 
Department of Mathematics and Statistics 
\\
University of Victoria
\\
Victoria, British Columbia 
\\
E-mail:ibrahim@math.uvic.ca
}
\\
\\
{
Hiroaki Kikuchi,
\\ 
School of Information Environment
\\
Tokyo Denki University, 
\\
Inzai, Chiba 270-1382, Japan 
\\
E-mail: hiroaki@sie.dendai.ac.jp
}
\\
\\
{
Hayato Nawa,
\\
Division of Mathematical Science, Department of System Innovation
\\
Graduate School of Engineering Science
\\
Osaka University, 
\\
Toyonaka 560-8531, Japan 
\\
E-mail: nawa@sigmath.es.osaka-u.ac.jp
}

\end{document}